\documentclass[12pt,leqno,draft]{article}
\usepackage{amsfonts}
\usepackage{amsmath, amsthm, amsfonts, amssymb, color}
\usepackage{mathrsfs}
\usepackage{color}
\usepackage{stmaryrd}
\makeatletter
\newcommand{\Spvek}[2][r]{%
  \gdef\@VORNE{1}
  \left(\hskip-\arraycolsep%
    \begin{array}{#1}\vekSp@lten{#2}\end{array}%
  \hskip-\arraycolsep\right)}

\def\vekSp@lten#1{\xvekSp@lten#1;vekL@stLine;}
\def\vekL@stLine{vekL@stLine}
\def\xvekSp@lten#1;{\def\temp{#1}%
  \ifx\temp\vekL@stLine
  \else
    \ifnum\@VORNE=1\gdef\@VORNE{0}
    \else\@arraycr\fi%
    #1%
    \expandafter\xvekSp@lten
  \fi}
\makeatother

\newtheorem{thm}{Theorem}[section]
\newtheorem{cor}[thm]{Corollary}

\newtheorem{rem}[thm]{Remark}
\theoremstyle{definition}

\newcommand{\scr}[1]{\mathscr #1}
\definecolor{wco}{rgb}{0.5,0.2,0.3}

\numberwithin{equation}{section} \theoremstyle{remark}

\newcommand{\ua}{\uparrow}

\title{{\bf    Harnack and Shift Harnack Inequalities for SDEs with Integrable Drifts}
}
\author{
{\bf     Xing Huang  }\\
\footnotesize{  Center for Applied Mathematics, Tianjin University, Tianjin 300072, China}\\
\footnotesize{  xinghuang@tju.edu.cn}}
\begin{document}
\allowdisplaybreaks
\def\R{\mathbb R}  \def\ff{\frac} \def\ss{\sqrt} \def\B{\mathbf
B}
\def\N{\mathbb N} \def\kk{\kappa} \def\m{{\bf m}}
\def\ee{\varepsilon}\def\ddd{D^*}
\def\dd{\delta} \def\DD{\Delta} \def\vv{\varepsilon} \def\rr{\rho}
\def\<{\langle} \def\>{\rangle} \def\GG{\Gamma} \def\gg{\gamma}
  \def\nn{\nabla} \def\pp{\partial} \def\E{\mathbb E}
\def\d{\text{\rm{d}}} \def\bb{\beta} \def\aa{\alpha} \def\D{\scr D}
  \def\si{\sigma} \def\ess{\text{\rm{ess}}}
\def\beg{\begin} \def\beq{\begin{equation}}  \def\F{\scr F}
\def\Ric{\text{\rm{Ric}}} \def\Hess{\text{\rm{Hess}}}
\def\e{\text{\rm{e}}} \def\ua{\underline a} \def\OO{\Omega}  \def\oo{\omega}
 \def\tt{\tilde} \def\Ric{\text{\rm{Ric}}}
\def\cut{\text{\rm{cut}}} \def\P{\mathbb P} \def\ifn{I_n(f^{\bigotimes n})}
\def\C{\scr C}   \def\G{\scr G}   \def\aaa{\mathbf{r}}     \def\r{r}
\def\gap{\text{\rm{gap}}} \def\prr{\pi_{{\bf m},\varrho}}  \def\r{\mathbf r}
\def\Z{\mathbb Z} \def\vrr{\varrho} \def\ll{\lambda}
\def\L{\scr L}\def\Tt{\tt} \def\TT{\tt}\def\II{\mathbb I}
\def\i{{\rm in}}\def\Sect{{\rm Sect}}  \def\H{\mathbb H}
\def\M{\scr M}\def\Q{\mathbb Q} \def\texto{\text{o}} \def\LL{\Lambda}
\def\Rank{{\rm Rank}} \def\B{\scr B} \def\i{{\rm i}} \def\HR{\hat{\R}^d}
\def\to{\rightarrow}\def\l{\ell}\def\iint{\int}
\def\EE{\scr E}\def\no{\nonumber}
\def\A{\scr A}\def\V{\mathbb V}\def\osc{{\rm osc}}
\def\BB{\scr B}\def\Ent{{\rm Ent}}\def\3{\triangle}\def\H{\scr H}
\def\U{\scr U}\def\8{\infty}\def\1{\lesssim}\def\HH{\mathrm{H}}
 \def\T{\scr T}
\maketitle

\begin{abstract} In this paper, the coupling by change of measure is constructed for a class of SDEs with integrable drift and additive noise, from which the Harnack and shift Harnack inequalities are derived. Finally, as applications, the gradient estimate, the regularity of the heat kernel and the distribution properties of the associated transition probability are also obtained. The important tool is Krylov's estimate.
\end{abstract} \noindent
 \vskip 2cm

\section{Introduction}
 Let $E$ be a topological space, $P$ is a Markov operator on $\B_b(E)$ (the bounded measurable functions on $E$), the dimension-free Harnack inequality with power $p$, i.e.
\begin{equation}\label{Har}(P f)^p(x)\leq P f^p(y) \e ^{\Psi(x,y)}, \ \ x, y \in E, f\in \B^+_b(E)\end{equation}
 has many applications, for instance, it implies a dimension-free lower bound for logarithmic Sobolev constant on compact manifolds \cite{FYW0}. It also yields strong Feller property, gradient estimate, uniqueness of invariant probability, regularity of the heat kernel with respect to invariant probability, see \cite[Chapter 1]{Wbook}. Moreover, it is an important tool in the proof of hypercontractivity of non-symmetric semigroup, \cite{BWY15,W2}. On the other hand, when $E$ is a Banach space, the shift Harnack inequality
 \begin{equation}\label{shHar}\Phi(P f(x))\leq P\{ \Phi\circ f(y+\cdot)\} \e ^{C_{\Phi}(x,y)}, \ \ x, y \in E, f\in \B^+_b(E)\end{equation}
 implies the existence and regularity of density of $P$ with respect to the Lebesgue measure. Thus, the Harnack and shift Harnack inequalities attracts much attention and there are many results on this topic, of which \cite{Wbook} gives lots of models satisfying Harnack and shift Harnack inequalities. For simplicity, consider the SDE on $\mathbb{R}^d$ below:
 \beq\label{EH0}
\d X_t=b_t(X_t)\d t+\d W_t.
\end{equation}
The classical condition for Harnack and shift Harnack inequalities is
\begin{equation}\label{bh}\langle b_t(x)-b_t(y), x-y\rangle\leq C|x-y|^2, \ \ t\geq 0, x,y \in\mathbb{R}^d\end{equation}
for a constant $C>0$.
 Recently, Zvonkin type transforms have been used to prove existence and uniqueness of SDEs and SPDEs with singular drift, see e.g. \cite{B,GM,H,HW,HW2,HW3,P,FYW,WZ,WZ15,Z,Z2,ZV}. Following \cite{Z2}, Shao \cite{Shao} proved the Harnack inequality \eqref{Har} under the condition $|b|+|b|^2\in L^q_p(T)$ for some $p,q>1$ satisfying $\frac{d}{p}+\frac{2}{q}<1$, here $L^q_p(T)$ is defined in \eqref{Lpq}. However, the Harnack inequality in \cite{Shao} is not precise since $\lim_{y\to x}\e^{\Psi(x,y)}>1$. In addition, \cite{LLW} has obtained the precise log-Harnack inequality by gradient-gradient estimate
 $$|\nabla P_t f|^2\leq C P_t |\nabla f|^2$$
by approximation method when $|b|\in L^q_p(T)$ for some $p,q>1$ satisfying $\frac{d}{p}+\frac{2}{q}<1$, which removes the condition $|b|^2\in L^q_p(T)$ in \cite{Shao}. Unfortunately, \cite{LLW} can not obtain gradient-gradient estimate
\begin{align}\label{gra}
|\nabla P_t f|\leq C P_t |\nabla f|,\end{align}
which implies the precise Harnack inequality \eqref{Har} by \cite[Theorem 1.3.6 (2)]{Wbook}. To obtain precise Harnack inequality \eqref{Har} in the sense that $\lim_{y\to x}\e^{\Psi(x,y)}=1$, instead of proving \eqref{gra}, we adopt the method of coupling by change of measure. To this end, we introduce an additional condition \eqref{b-co} below, which means $b$ satisfying $$\sup_{y\neq 0}\frac{\|b(\cdot+y)-b\|_{p}}{|y|}<\infty$$ for some $p>d$ when $q=\infty$, where $\|\cdot\|_p$ is the $L^p$ norm respect to Lebesgue measure, see Remark \ref{example} for example and more details.

Compared with the existed precise Harnack inequalities, the drift in this paper is allowed to be integrable and not continuous. As to the shift Harnack inequality, it is very new since there is few result for SDE with integrable drift on this topic.

Throughout the paper, the letter $C$ or $c$ will denote a positive constant, and $C(\theta)$ or $c(\theta)$ stands for a constant depending on $\theta$. The value of the constants may change from one appearance to another.

For a measurable function $f$ defined on $[0,T]\times\mathbb{R}^d$, let
\begin{equation}\label{Lpq}\|f\|_{L^q_p(s,t)}=\left(\int_s^t\left(\int_{\mathbb{R}^d}|f_r(x)|^p\d x\right)^{\frac{q}{p}}\d r\right)^{\frac{1}{q}}, \ \ p,q\ge 1, 0\le s\le t\le T. \end{equation} When $s=0$, we simply denote   $\|f\|_{L^q_p(0,t)}=\|f\|_{L^q_p(t)}$.
Let $W_t$ be an $m$-dimensional Brownian motion on a complete
filtration probability space $(\OO, \{\F_{t}\}_{t\ge 0},\P)$.
Consider the following SDEs on $\mathbb{R}^{d}$:
\beq\label{EH}
\d X_t=b_t(X_t)\d t+\sigma_t\d W_t,
\end{equation}
where
$$b: [0,\infty)\times \R^d\to \R^d;\ \ \si: [0,\infty)\times \R^d\to  \R^d\otimes\R^m$$
are   measurable.
Throughout this paper, we make the following assumptions:
\beg{enumerate} \item[{\bf (H1)}] There exists constants $p,q> 1$ with $\frac{d}{p}+\frac{2}{q}<1$ such that
 \begin{align}\label{b-in}
\|b\|_{L_p^q(T)}<\infty,\ \ T\geq0.
\end{align}
Moreover, there exists an nonnegative function $K\in L^{q}_{loc}([0,\infty))$ such that
 \begin{align}\label{b-co}
\left(\int_{\mathbb{R}^d}|b_t(x+y)-b_t(x)|^{p}\d x\right)^{\frac{1}{p}}\leq K(t)|y|,\ \ t\geq0,y\in\mathbb{R}^d.
\end{align}
\item[{\bf (H2)}] There exists a constant $\delta\in(1,\infty)$ such that for any $t\in[0,\infty)$,
$$\delta^{-1} I_{d\times d}\leq\sigma_t\sigma_t^\ast\leq\delta I_{d\times d}.$$
%\item[{\bf (H3)}] For any $T\geq 0$,
%\begin{align}\label{sigma}
%\|\nabla\sigma\|_{L^q_p(T)}<\infty,
%\end{align}
%where $\nabla $ is the weak gradient.
\end{enumerate}
According to \cite[Theorem 1.1]{Z2}, under \eqref{b-in} and {\bf(H2)}, the equation \eqref{EH} has a unique non-explosive strong solution $X_t^x$ with $X_0=x\in\mathbb{R}^d$. Let $P_t$ be the associated Markov semigroup, i.e.
$$P_tf(x)=\mathbb{E}f(X_t^x), \ \ f\in\B_b(\mathbb{R}^d).$$
\begin{rem}\label{example} To obtain precise Harnack inequality, we introduce \eqref{b-co} instead of the Lipschitzian continuity for $b$, i.e.
\begin{align}\label{L-c}\|b_t(\cdot+y)-b_t(\cdot)\|_{\infty}<C(t)|y|.
\end{align} To see the difference between \eqref{b-co} and \eqref{L-c}, we give an example as follows. Let $b=a1_{[c_1,c_2]}$ for $a, c_1,c_2\in\mathbb{R}^d$ with $a\neq0$ and $c_1\leq c_2$. Obviously, $b$ does not satisfy \eqref{L-c} (in fact, $b$ does not satisfy \eqref{bh} either), but by a simple calculus, \eqref{b-co} holds.  From this example, we see that $b$ may be not continuous if \eqref{b-co} holds. On the other hand, it is well known that
$$\|f(\cdot+y)-f\|_{p}\leq \|\nabla f\|_{p}|y|, \ \ p>0, $$
where $\nabla $ is the weak gradient. This means that if $\|\nabla b_t\|_{p}<K(t)$ for some $K\in L^{q}_{loc}([0,\infty))$, then \eqref{b-co} holds. %Since $p>d$, by Sobolev embedding theorem, for any $\alpha<1-d/p$,
%$$\sup_{x\neq y}\frac{|b(x)-b(y)|}{|x-y|^\alpha}\leq C(\|\nabla b\|_{L^p}+\|b\|_{L^p}).$$

\end{rem}
Let  $$\scr K:=\Big\{(p,q)\in (1,\infty)\times(1,\infty):\   \ff d p +\ff 2 q<2\Big\}.$$
We firstly give an important lemma which will be used in the sequence.
\beg{lem}\label{KK} Let $T>0$. Assume \eqref{b-in} and {\bf(H2)}.
 Then for any $(\aa,\bb)\in \scr K$, there exists a constant  $\kappa=\kappa(T,\delta, \aa,\bb, \|b\|_{L_p^q(T)})>0$ such that for any $s_0\in [0,T)$ and any solution $(X_{s_0,t})_{t\in [s_0,T]}$ of $\eqref{EH}$ from time $s_0$,
\beq\label{APP'}\E\bigg[\int_s^t |f|(r, X_{s_0,r}) \d r\Big| \F_s\bigg]\le   \kappa\|f\|_{L_{\aa}^{\bb}(T)},\ s_0\leq s<t\leq T, f\in L_{\aa}^{\bb}(T).\end{equation}
Then for any $\lambda>0$, there exists a constant $\gamma=\gamma(\lambda, \kappa,\|f\|_{L_{\aa}^{\bb}(T)})>0$ such that
\begin{align}\label{Kh}\E\big(\e^{\ll\int_s^T |f|(r,X_{s_0,r}) \d r}\big| \F_s\big)\leq\gamma, \ \ s_0\leq s\leq T.\end{align}
Moreover,
\beq\label{KR2}   \E\big(\e^{\ll\int_s^T |f|(r,X_{s_0,r}) \d r}\big| \F_s\big) \le \frac{1}{1-\lambda \kappa\|f\|_{L_{\aa}^{\bb}(T)}},\ \ s_0\leq s\leq T\end{equation}
when $\|f\|_{L_{\aa}^{\bb}(T)}<\frac{1}{\lambda\kappa}$.
 \end{lem}
 \begin{proof} \eqref{APP'}, which is called Krylov's estimate, was proved in \cite[Lemma 3.3]{KR}, see also \cite[Lemma 3.1]{HW3} for the multiplicative noise case. \eqref{Kh} follows from \eqref{APP'} and Khasminskii's estimate. We only need to prove \eqref{KR2}. Since \eqref{APP'} implies that for any $n\geq 1$, $\lambda>0$,
  \beq\label{APP''}\lambda^n\E\bigg[\left(\int_s^T |f|(r, X_{s_0,r}) \d r\right)^n\Big| \F_s\bigg]\le   n!  \left(\lambda \kappa\|f\|_{L_{\aa}^{\bb}(T)}\right)^n,\ \ s_0\leq s\leq T, f\in L_{\aa}^{\bb}(T).\end{equation}
 Thus, if $\|f\|_{L_{\aa}^{\bb}(T)}<\frac{1}{\lambda \kappa}$, we have
 \beq\label{KR2'}   \E\big(\e^{\ll\int_s^T|f|(r,X_{s_0,r}) \d r}\big| \F_s\big) \le \sum^\infty_{n=0}\left(\lambda \kappa\|f\|_{L_{\aa}^{\bb}(T)}\right)^n=\frac{1}{1-\lambda\kappa\|f\|_{L_{\aa}^{\bb}(T)}},\ \ s\in[s_0,T]. \end{equation}
% Taking $s=t_0\leq t_1<\cdots<t_k=T$ with $t_i-t_{i-1}=\frac{T}{k}\leq (2\lambda\|f\|_{L_{\aa}^{\bb}(T)})^{-\frac{1}{\gamma}}$, we have
% \beq\label{KR2''}\begin{split}   \E\big(\e^{\ll\int_s^T f(r,X_{s,r}) \d r}\big| \F_s\big) &=\E\left(\prod_{i=0}^{k-1}\e^{\ll\int_{t_i}^{t_{i+1}} f(r,X_{s,r}) \d r}\Big| \F_s\right)\\ &=\E\left(\E\left(\e^{\ll\int_{t_{k-1}}^{t_{k}} f(r,X_{s,r}) \d r}\Big| \F_{t_{k-1}}\right)\prod_{i=0}^{k-2}\e^{\ll\int_{t_i}^{t_{i+1}} f(r,X_{s,r}) \d r}\Big| \F_s\right)\\
% &\leq C\E\left(\prod_{i=0}^{k-2}\e^{\ll\int_{t_i}^{t_{i+1}} f(r,X_{s,r}) \d r}\Big| \F_s\right)\\
% &\leq C^k ,\ \ s\in[0,T]. \end{split}\end{equation}
 \end{proof}
 The paper is organized as follows: In Section 2, we give main results on Harnack and shift Harnack inequality and their applications respectively; In Section 3, we prove Harnack inequality; In Section 4, we prove shift Harnack inequality.
\section{Main Results}
\subsection{Harnack Inequality and Its Applications}
\begin{thm}\label{T-Har}  Assume
  {\bf (H1)}-{\bf (H2)}. Let $T>0$ and $\beta(T,K,\delta,\kappa)=\delta\left(\frac{1}{T}+\kappa\|K\|^2_{L^{q}([0,T])}\right)$.
%  Then for any $x,y\in \R^{d}$ and positive
%$f\in \B_b(\R^{d})$, %the log-Harnack inequality holds, i.e.
%\beg{equation*}\beg{split}
%&P_T\log f(y)\leq\log P_T f(x)+ \beta(T,K,\delta)|x-y|^2
%,\end{split}\end{equation*}
 Then for any nonnegative
$f\in \B_b(\R^{d})$ and any $p>1$,
 \beg{equation*}\beg{split}  (P_Tf)^p(y)\le  &P_Tf^p(x)
  \left (1-\frac{(2p+2)\beta(T,K,\delta,\kappa)|x-y|^2}{(p-1)^2}\right)^{-\frac{p-1}{2}}
  \end{split}\end{equation*}
  holds for any $x,y\in \R^{d}$ with $|x-y|^2<\left(\frac{(2p+2)\beta(T,K,\delta,\kappa)}{(p-1)^2}\right)^{-1}$.
 \end{thm}
The next corollary following from Theorem \ref{T-Har} describes the property of the transition probability, see \cite[Theorem 1.4.2 (1)]{Wbook} for the proof.
\begin{cor}\label{density0} Let the assumption in Theorem \ref{T-Har} hold. Let $T>0$, $p>1$, $x,y\in\mathbb{R}^d$ with $|x-y|^2<\left(\frac{(2p+2)\beta(T,K,\delta,\kappa)}{(p-1)^2}\right)^{-1}$.  Then $P_T(x,\cdot)$ is equivalent to $P_T(y,\cdot)$ and
$$P_T\left\{\left(\frac{\d P_T(x,\cdot)}{\d P_T(y,\cdot)}\right)^{\frac{1}{p-1}}\right\}(x)\leq \left (1-\frac{(2p+2)\beta(T,K,\delta,\kappa)|x-y|^2}{(p-1)^2}\right)^{-\frac{1}{2}}.$$
\end{cor}
\subsection{Shift Harnack Inequality and Its Applications}
The following theorem gives the result on the shift Harnack inequality.
\begin{thm}\label{T-shift}  Let $T>0$. Assume \eqref{b-in}
   and {\bf (H2)}. Then the following assertions hold.
\begin{enumerate}
\item[(i)] For any $x,y\in \R^{d}$ and positive
$f\in \B_b(\R^{d})$, the shift log-Harnack inequality holds, i.e.
\beg{equation*}\beg{split}
&P_T\log f(x)\leq\log P_T f(y+\cdot)(x)+ \delta\left(\frac{|y|^2}{T}+4\kappa\|b\|^2_{L_p^{q}(T)}\right).\end{split}\end{equation*}
Moreover, for any $p>1$, and any nonnegative
$f\in \B_b(\R^{d})$, it holds that
 \beg{equation*}\beg{split}  (P_Tf)^p(x)\le  &P_Tf^p(y+\cdot)(x)  \gamma\e ^{\frac{\delta(p+1)|y|^2}{2(p-1)T}}\end{split}\end{equation*}
with some constant $\gamma>0$ depending on $p, \kappa,\delta, \|b\|^2_{L_p^{q}(T)}$.
\item[(ii)] If in addition \eqref{b-co} holds, then for any $x,y\in \R^{d}$ and positive
$f\in \B_b(\R^{d})$, the shift log-Harnack inequality holds, i.e.
\beg{equation*}\beg{split}
&P_T\log f(x)\leq\log P_T f(y+\cdot)(x)+ \beta(T,K,\delta,\kappa)|y|^2.\end{split}\end{equation*}
Moreover, for any $p>1$, and any nonnegative
$f\in \B_b(\R^{d})$,
 \beg{equation*}\beg{split}  (P_Tf)^p(x)\le  &P_Tf^p(y+\cdot)(x)
  \left (1-\frac{(2p+2)\beta(T,K,\delta,\kappa)|y|^2}{(p-1)^2}\right)^{-\frac{p-1}{2}}\end{split}\end{equation*}
  holds for any $x,y \in\mathbb{R}^d$ with $|y|^2<\left(\frac{(2p+2)\beta(T,K,\delta,\kappa)}{(p-1)^2}\right)^{-1}$. Here, $\beta(T,K,\delta,\kappa)$ is defined in Theorem \ref{T-Har}.
\end{enumerate}
 \end{thm}
According to \cite[Theorem 1.4.3, Proposition 1.3.9 (2)]{Wbook}, we have the following corollary from Theorem \ref{T-shift}.
\begin{cor}\label{C-density} Let $T>0$. Assume \eqref{b-in}
   and {\bf (H2)}.
\begin{enumerate}
\item[(i)]  For any $x,y\in \R^{d}$,  $P_T$ has transition density $p_T(x,y)$ with respect to the Lebesgue measure such that
$$\int_{\mathbb{R}^{d}}p_T(x,y)^{\frac{p}{p-1}}\d y\leq \frac{1}{\left(\gamma^{-1}\int_{\mathbb{R}^{d}}\e ^{-\frac{\delta(p+1)|y|^2}{2(p-1)T}}\d y\right)^{\frac{1}{p-1}}}$$
for any $p>1$ and some constant $\gamma>0$ depending on $p, \kappa,\delta, \|b\|^2_{L_p^{q}(T)}$.
\item[(ii)] If in addition \eqref{b-co} holds, then
$$|P_T(\nabla_yf)(x)|^2\leq 2\beta(T,K,\delta,\kappa)\{P_Tf^2(x)-(P_Tf)^2(x)\}, \ \ x,y \in\mathbb{R}^d, f \in C^1_b(\mathbb{R}^d).$$
\end{enumerate}
Moreover, for any $p>1$, $x,y \in\mathbb{R}^d$ with $|y|^2<\left(\frac{(2p+2)\beta(T,K,\delta,\kappa)}{(p-1)^2}\right)^{-1}$, $P_T(x,\cdot)$ is equivalent to $P_T(x,\cdot-y)$ and
$$P_T\left\{\left(\frac{\d P_T(x,\cdot)}{\d P_T(x,\cdot-y)}\right)^{\frac{1}{p-1}}\right\}(x)\leq \left (1-\frac{(2p+2)\beta(T,K,\delta,\kappa)|y|^2}{(p-1)^2}\right)^{-\frac{1}{2
}}.$$
\end{cor}
\section{Proof of Theorem \ref{T-Har}}
We use the coupling by change of measure to derive the Harnack inequality.
\begin{proof}[Proof of Theorem \ref{T-Har}]
For any $x\in\mathbb{R}^d$, let $X^x_t$ solve (\ref{EH})  with $X_0= x$,
and $Y_t$ solve the equation
\beq\label{EC1} \d Y_t= b_t(X_t^x)\d t +\si_t \d W_t+ \frac{x-y}{T}\d t\end{equation} with
$Y_0= y$. Then we have

\beq\label{EE} Y_s=X^x_s+\frac{(s-T)(x-y)}{T}, \ \
s\in[0,T].\end{equation}
In particular, $X_T^x=Y_T$. Set
\begin{align*}
R(s)=\exp\bigg[-\int_0^s\< \si_u^\ast(\si_u\si^\ast_u)^{-1}\Phi(u), \d
W_u\>-\frac{1}{2}\int_0^s |\si_u^\ast(\si_u\si_u^\ast)^{-1}\Phi(u)|^2\d u\bigg],
\end{align*}
and
$$
\bar{W}_s=W_s+\int_0^s\si_u^\ast(\si_u\si^\ast_u)^{-1}\Phi(u)\d u,
$$
where
$$
\Phi(s)=b_s(X_s^x)-b_s(Y_s)+\frac{x-y}{T}.
$$
By Lemma \ref{KK} for $X_t^x$ and $\alpha=p/2$, $\beta=q/2$, \eqref{b-in} and \eqref{EE} imply that
\begin{align*}\mathbb{E}\int_0^T \left|b_s(X_s^x)-b_s(Y_s)\right|^2\d s&=\mathbb{E}\int_0^T \left|b_s(X_s^x)-b_s\left(X^x_s+\frac{(s-T)(x-y)}{T}\right)\right|^2\d s\\
&\leq \kappa\left\{\int_{0}^T\left(\int_{\mathbb{R}^d}\left|b_s(z+\frac{(s-T)(x-y)}{T})-b_s(z)\right|^{p}\d z\right)^{\frac{q}{p}}\d s\right\}^{\frac{2}{q}}\\
&\leq 4\kappa\|b\|^2_{L^q_p(T)}.
\end{align*}
Then by \eqref{Kh} and {\bf(H2)}, we have
$$\mathbb{E}\exp\left\{\int_0^T |\si_u^\ast(\si_u\si_u^\ast)^{-1}\Phi(u)|^2\d u\right\}<\infty.$$
By Girsanov's theorem, $\{\bar{W}_s\}_{s\in[0,T]}$ is a Brownian motion under $\Q_T=R(T)\P$.
Then (\ref{EC1}) reduces to
\beq\label{E2'}
\d Y_t= b_t(Y_t)\d t +\si_t \d \bar{W}_t,
\end{equation}
which together with the weak uniqueness of \eqref{EH} implies the distribution of $Y_T$ under $\Q_T$ coincides with the one of $X_T^y$  under $\P$.
Thus, from \eqref{EE}, \eqref{b-co}, and \eqref{APP'} with $\alpha=p/2$, $\beta=q/2$, it holds that
\begin{equation}\begin{split}\label{Phi0}
&\E^{\Q_T}\int_0^T|\Phi(s)|^2\d s\\
&\leq 2\E^{\Q_T}\int_0^T\left(\frac{|x-y|^2}{T^2}+\left|b_s(X_s^x)-b_s(Y_s)\right|^2\right)\d s\\
&= 2\frac{|x-y|^2}{T}+2\E^{\Q_T}\int_0^T\left|b_s(X_s^x)-b_s(Y_s)\right|^2\d s\\
&=2\frac{|x-y|^2}{T}+2\mathbb{E}^{\Q_T}\int_0^T \left|b_s(Y_s)-b_s\left(Y_s+\frac{(T-s)(x-y)}{T}\right)\right|^2\d s\\
&\leq2\frac{|x-y|^2}{T}+ 2\kappa\left\{\int_{0}^T\left(\int_{\mathbb{R}^d}\left|b_s(z+\frac{(T-s)(x-y)}{T})-b_s(z)\right|^{p}\d z\right)^{\frac{q}{p}}\d s\right\}^{\frac{2}{q}}\\
&\leq2\frac{|x-y|^2}{T}+2\kappa\left\{\int_0^TK(s)^{q}\d s\right\}^{\frac{2}{q}}|x-y|^2 \\
&\leq 2\left(\frac{1}{T}+\kappa\|K\|^2_{L^{q}([0,T])}\right)|x-y|^2.
\end{split}\end{equation}
%By Young's inequality,
%\begin{align*}
%P_T \log f(y)&=\E^{\Q_T}\log f(Y_T)=\E ^{\Q_T}\log f(X_T^x)\\
%&\leq \log P_T f(x)+\E R(T)\log R(T),
%\end{align*}
By H\"{o}lder inequality, we have
\begin{align}\label{Holder}
P_T f(y)&=\E^{\Q_T}f(Y_T)=\E ^{\Q_T}f(X_T^x)\leq (P_T f^p(x))^{\frac{1}{p}}\{\E R(T)^{\frac{p}{p-1}}\}^{\frac{p-1}{p}}.
\end{align}
Combining \eqref{Phi0} and {\bf(H2)}, it follows
%\begin{align*}
%\E R(T)\log R(T)=\E ^{\Q_T}\log R(T)=&\frac{1}{2}\E^{\Q_T} \int_0^T |(\si_u^\ast(\si_u\si^\ast_u)^{-1}\Phi(u)|^2\d u\\
%&\leq\beta(T,K)|x-y|^2.
%\end{align*}
from H\"{o}lder inequality and \eqref{KR2} that
\begin{align*}
\left(\E R(T)^{\frac{p}{p-1}}\right)^2&\leq\E^{\Q_T}\exp\left\{\frac{p+1}{(p-1)^2}\int_0^T |(\si_u^\ast(\si_u\si^\ast_u)^{-1}\Phi(u)|^2\d u\right\}\\
&\leq \frac{1}{1-\frac{(2p+2)\beta(T,K,\delta,\kappa)|x-y|^2}{(p-1)^2}}
\end{align*}
if $|x-y|^2<\left(\frac{(2p+2)\beta(T,K,\delta,\kappa)}{(p-1)^2}\right)^{-1}$.
Substituting this into \eqref{Holder}, we complete the proof.
\end{proof}
\section{Proof of Theorem \ref{T-shift}}
\begin{proof}[Proof of Theorem \ref{T-shift}]
For any $x\in\mathbb{R}^d$, let $X^x_t$ solve (\ref{EH})  with $X_0= x$,
and $\tilde{Y}_t$ solve the equation
\beq\label{EC1'} \d \tilde{Y}_t= b_t(X_t^x)\d t +\si_t \d W_t+ \frac{y}{T}\d t\end{equation} with
$\tilde{Y}_0= x$. Then we have
\beq\label{EE'} \tilde{Y}_s=X^x_s+\frac{s}{T}y, \ \
s\in[0,T].\end{equation}
In particular, $X_T^x+y=\tilde{Y}_T$. Let
\begin{align*}
\tilde{R}(s)=\exp\bigg[-\int_0^s\< \si_u^\ast(\si_u\si^\ast_u)^{-1}\tilde{\Phi}(u), \d
W_u\>-\frac{1}{2}\int_0^s |\si_u^\ast(\si_u\si_u^\ast)^{-1}\tilde{\Phi}(u)|^2\d u\bigg],
\end{align*}
and
$$
\tilde{W}_s=W_s+\int_0^s\si_u^\ast(\si_u\si^\ast_u)^{-1}\tilde{\Phi}(u)\d u,
$$
where $$\tilde{\Phi}(s)=b_s(X_s^x)-b_s(\tilde{Y}_s)+\frac{y}{T}.$$
Again by Lemma \ref{KK} for $X_t^x$ and $\alpha=p/2$, $\beta=q/2$, it follows from \eqref{b-in} and \eqref{EE'} that
\begin{align*}&\mathbb{E}\int_0^T \left|b_s(X_s^x)-b_s(\tilde{Y}_s)\right|^2\d s\\
&=\mathbb{E}\int_0^T \left|b_s(X_s^x)-b_s\left(X^x_s+\frac{s}{T}y\right)\right|^2\d s\\
&\leq \kappa\left\{\int_{0}^T\left(\int_{\mathbb{R}^d}\left|b_s(z+\frac{s}{T}y)-b_s(z)\right|^{p}\d z\right)^{\frac{q}{p}}\d s\right\}^{\frac{2}{q}}\\
&\leq \kappa\left\{\int_{0}^T\left(\left(\int_{\mathbb{R}^d}\left|b_s(z+\frac{s}{T}y)\right|^{p}\d z\right)^{\frac{1}{p}}+\left(\int_{\mathbb{R}^d}\left|b_s(z)\right|^{p}\d z\right)^{\frac{1}{p}}\right)^{q}\d s\right\}^{\frac{2}{q}}\\
&\leq 4\kappa\|b\|^2_{L^q_p(T)}.
\end{align*}
Then by \eqref{Kh} and {\bf(H2)}, we have
$$\mathbb{E}\exp\left\{\int_0^T |\si_u^\ast(\si_u\si_u^\ast)^{-1}\tilde{\Phi}(u)|^2\d u\right\}<\infty.$$
Applying Girsanov's theorem, we obtain that $\{\tilde{W}_s\}_{s\in[0,T]}$ is a Brownian motion under $\tilde{\Q}_T=\tilde{R}(T)\P$.
Then (\ref{EC1'}) reduces to
\beq\label{E2''}
\d \tilde{Y}_t= b_t(\tilde{Y}_t)\d t +\si_t \d \tilde{W}_t,
\end{equation}
and this together with the weak uniqueness of \eqref{EH} yields the distribution of $\tilde{Y}_T$ under $\tilde{\Q}_T$ coincides with the one of $X_T^y$  under $\P$.
By Young's inequality,
\begin{align*}
P_T \log f(x)&=\E^{\tilde{\Q}_T}\log f(\tilde{Y}_T)=\E ^{\tilde{\Q}_T}\log f(X_T^x+y)\leq \log P_T f(y+\cdot)(x)+\E \tilde{R}(T)\log \tilde{R}(T),
\end{align*}
and by H\"{o}lder inequality,
\begin{align*}
P_T f(x)&=\E^{\tilde{\Q}_T}f(\tilde{Y}_T)=\E ^{\tilde{\Q}_T}f(X_T^x+y)\leq (P_T f^p(y+\cdot))^{\frac{1}{p}}(x)\{\E \tilde{R}(T)^{\frac{p}{p-1}}\}^{\frac{p-1}{p}}.
\end{align*}
(i) \eqref{b-in}, \eqref{APP'} yield that
\begin{equation}\begin{split}\label{Phi'}
\E^{\tilde{\Q}_T}\int_0^T|\tilde{\Phi}(s)|^2\d s&\leq 2\E^{\tilde{\Q}_T}\int_0^T\left(\frac{|y|^2}{T^2}+\left|b_s(X_s^x)-b_s(\tilde{Y}_s)\right|^2\right)\d s\\
&= 2\frac{|y|^2}{T}+2\E^{\tilde{\Q}_T}\int_0^T\left(\left|b_s(X_s^x)-b_s(\tilde{Y}_s)\right|^2\right)\d s\\
&=2\frac{|y|^2}{T}+2\mathbb{E}^{\tilde{\Q}_T}\int_0^T \left|b_s(\tilde{Y}_s)-b_s\left(\tilde{Y}_s+\frac{s}{T}y\right)\right|^2\d s\\
&\leq2\frac{|y|^2}{T}+ 8\kappa\|b\|^2_{L^q_p(T)}.
\end{split}\end{equation}

Combining \eqref{Phi'} and {\bf(H2)}, we arrive at
\begin{align*}
\E \tilde{R}(T)\log \tilde{R}(T)=\E ^{\tilde{\Q}_T}\log \tilde{R}(T)=&\frac{1}{2}\E^{\tilde{\Q}_T} \int_0^T |(\si_u^\ast(\si_u\si^\ast_u)^{-1}\tilde{\Phi}(u)|^2\d u\\
&\leq\delta\left(\frac{|y|^2}{T}+4\kappa\|b\|^2_{L_p^{q}(T)}\right).
\end{align*}
It follows from H\"{o}lder inequality and \eqref{Kh} that
\begin{align*}
\left(\E \tilde{R}(T)^{\frac{p}{p-1}}\right)^2&\leq\E^{\tilde{\Q}_T}\exp\left\{\frac{p+1}{(p-1)^2}\int_0^T |(\si_u^\ast(\si_u\si^\ast_u)^{-1}\tilde{\Phi}(u)|^2\d u\right\}\\
&\leq \gamma\e ^{\frac{\delta(p+1)|y|^2}{T(p-1)^2}}
\end{align*}
with $\gamma$ depending on $\kappa,\delta, \|b\|^2_{L_p^{q}(T)}$. Thus, we finish the proof of (i).

(ii) If moreover \eqref{b-co} holds, then \eqref{b-co} and \eqref{APP'} yield that
\begin{equation}\begin{split}\label{Phi}
\E^{\tilde{\Q}_T}\int_0^T|\tilde{\Phi}(s)|^2\d s&\leq 2\E^{\tilde{\Q}_T}\int_0^T\left(\frac{|y|^2}{T^2}+\left|b_s(X_s^x)-b_s(\tilde{Y}_s)\right|^2\right)\d s\\
&= 2\frac{|y|^2}{T}+2\E^{\tilde{\Q}_T}\int_0^T\left(\left|b_s(X_s^x)-b_s(\tilde{Y}_s)\right|^2\right)\d s\\
&=2\frac{|y|^2}{T}+2\mathbb{E}^{\tilde{\Q}_T}\int_0^T \left|b_s(\tilde{Y}_s)-b_s\left(\tilde{Y}_s+\frac{s}{T}y\right)\right|^2\d s\\
&\leq2\frac{|y|^2}{T}+ 2\kappa\left\{\int_{0}^T\left(\int_{\mathbb{R}^d}\left|b_s(z+\frac{s}{T}y)-b_s(z)\right|^{p}\d z\right)^{\frac{q}{p}}\d s\right\}^{\frac{2}{q}}\\
&\leq2\frac{|y|^2}{T}+2\kappa\left\{\int_0^TK(s)^{2q}\d s\right\}^{\frac{1}{q}}|y|^2 \\
&\leq 2\left(\frac{1}{T}+\kappa\|K\|^2_{L^{q}([0,T])}\right)|y|^2.
\end{split}\end{equation}

Combining \eqref{Phi} and {\bf(H2)}, we arrive at
\begin{align*}
\E \tilde{R}(T)\log \tilde{R}(T)=\E ^{\tilde{\Q}_T}\log \tilde{R}(T)=&\frac{1}{2}\E^{\tilde{\Q}_T} \int_0^T |(\si_u^\ast(\si_u\si^\ast_u)^{-1}\tilde{\Phi}(u)|^2\d u\\
&\leq\delta\left(\frac{1}{T}+\kappa\|K\|^2_{L^{q}([0,T])}\right)|y|^2.
\end{align*}
It follows from H\"{o}lder inequality and \eqref{KR2} that
\begin{align*}
\left(\E \tilde{R}(T)^{\frac{p}{p-1}}\right)^2&\leq\E^{\tilde{\Q}_T}\exp\left\{\frac{p+1}{(p-1)^2}\int_0^T |(\si_u^\ast(\si_u\si^\ast_u)^{-1}\tilde{\Phi}(u)|^2\d u\right\}\\
&\leq \frac{1}{1-\frac{(2p+2)\beta(T,K,\delta,\kappa)|y|^2}{(p-1)^2}}
\end{align*}
if $|y|^2<\left(\frac{(2p+2)\beta(T,K,\delta,\kappa)}{(p-1)^2}\right)^{-1}$.
Thus, the proof is completed.
 \end{proof}
\begin{rem} In fact, from the construction of the coupling by change of measure, we only use the weak existence and uniqueness of \eqref{EH}.  Thus, we may replace \eqref{b-in} by some weaker integrable condition that ensures the weak existence and uniqueness of \eqref{EH}.
\end{rem}
\paragraph{Acknowledgement.} The author would like to thank Professor Feng-Yu Wang for corrections and helpful comments.

\beg{thebibliography}{99}

\bibitem{B} K. Bahlali,  \emph{Flows of homeomorphisms of stochastic differential equations with measurable drift,}  Stochastic Rep. 67(1999), 53-82.

%\bibitem{BBMY} J. H. Bao, B. B\"{o}ttcher, X. R. Mao, C. G. Yuan,  \emph{Convergence rate of numerical solutions to SFDEs with jumps,}  Journal of Computational and Applied Mathematics, 236(2011), 119-131.

%\bibitem{BHY} J. H. Bao, X. Huang, C. G. Yuan \emph{Convergence rate of  Euler-Maruyama scheme for SDEs with rough coefficients,} ArXiv:1609.06080, (2016).

%\bibitem{BWY} J. Bao, F.-Y. Wang, C. Yuan, \emph{Derivative formula and Harnack inequality for degenerate functionals SDEs,} Stoch. Dyn. 13(2013), 943-951.

%\bibitem{BWY13} J. Bao, F.-Y. Wang, C. Yuan,  \emph{Bismut formulae and applications for functional SPDEs,} Bull. Sci. Math. 137 (2013), 509-522.

\bibitem{BWY15} J. Bao, F.-Y. Wang, C. Yuan,  \emph{Hypercontractivity for Functional Stochastic Differential Equations,} Stoch. Proc. Appl. 125(2015), 3636-3656

\bibitem{GM} L. Gy\"{o}ngy, T. Martinez,  \emph{On stochastic differential equations with locally unbounded drift,}  Czechoslovak Math. J. 51(2001), 763--783.

\bibitem{H} Xing Huang,  \emph{Strong Solutions for Functional SDEs with Singular Drift,}  to appear in Stochastic and Dynamics.

\bibitem{HW} X. Huang, F.-Y. Wang, \emph{Functional SPDE with Multiplicative Noise and Dini Drift,}  Ann. Fac. Sci. Toulouse Math. 6(2017), 519-537.

\bibitem{HW2} X. Huang, F.-Y. Wang, \emph{Degenerate SDEs with Singular Drift and Applications to Heisenberg Groups,}  to appear in J. Differential Equations.

\bibitem{HW3} X. Huang, F.-Y. Wang, \emph{Distribution Dependent SDEs with Singular Coefficients,}  arXiv:1805.01682.

\bibitem{KR} N. V. Krylov, M. R\"{o}ckner, \emph{Strong solutions of stochastic equations with singular time dependent drift,}  Probab. Theory Related Fields 131(2005), 154-196.

\bibitem{LLW} H. Li, D. Luo, J. Wang \emph{Harnack inequalities for SDEs with multiplicative noise and non-regular drift,} Stoch. Dyn. 15(2015), 18pp.

\bibitem{P} E. Priola, \emph{Pathwise Uniqueness for Singular SDEs driven by Stable Processes,} Osaka Journal of Mathematics, 49(2012), 421-447.

\bibitem{Shao} J. Shao, \emph{Harnack inequalities and heat kernel estimates for SDEs with singular drifts,} Bull. Sci. Math. 137(2013), 589-601.

\bibitem{FYW0} F.-Y. Wang,  \emph{Logarithmic Sobolev inequalities on noncompact Riemannian manifolds,}  Probab. Theory Related Fields 109(1997), 417-424.

\bibitem{FYW} F.-Y. Wang,  \emph{Gradient estimate and applications for SDEs in Hilbert space with multiplicative noise and Dini continuous drift,}  J. Differential Equations, 260(2016), 2792-2829.

\bibitem{Wbook} F.-Y. Wang, \emph{Harnack Inequality and Applications for Stochastic Partial Differential Equations,} Springer, New York, 2013.

\bibitem{W2} F.-Y. Wang, \emph{Hypercontractivity and Applications for Stochastic Hamiltonian Systems,} arXiv:1409.1995.

\bibitem{WZ} F.-Y. Wang, X. C. Zhang,  \emph{Degenerate SDE with H\"{o}lder-Dini Drift and Non-Lipschitz Noise Coefficient,} SIAM J. Math. Anal. 48 (2016), 2189-2226.

\bibitem{WZ15} F.-Y. Wang, X. C. Zhang,  \emph{Degenerate SDEs in Hilbert Spaces with Rough Drifts,} arXiv:1501.0415.

\bibitem{Z} X. C. Zhang,  \emph{Strong solutions of SDEs with singural drift and Sobolev diffusion coefficients,}  Stoch. Proc. Appl. 115(2005), 1805-1818.

\bibitem{Z2} X. Zhang, \emph{Stochastic homeomorphism flows of SDEs with singular drifts and Sobolev diffusion coefficients,} Electron. J. Probab. 16(2011), 1096--1116.

\bibitem{ZV} A. K. Zvonkin,  \emph{A transformation of the phase space of a diffusion process that removes the drift,}  Math. Sb. (1)93 (1974).
\end{thebibliography}

\end{document}